\documentclass[english,review]{article}
\usepackage[T1]{fontenc}
\usepackage[latin9]{inputenc}
\usepackage{prettyref}
\usepackage{mathrsfs}
\usepackage{mathtools}
\usepackage{amsmath}
\usepackage{amsthm}
\usepackage{amssymb}
\usepackage{setspace}
\usepackage{esint}
\makeatletter
\theoremstyle{plain}
\newtheorem{thm}{\protect\theoremname}[section]
  \theoremstyle{plain}
  \newtheorem{cor}[thm]{\protect\corollaryname}
  \theoremstyle{plain}
  \newtheorem{lem}[thm]{\protect\lemmaname}

\usepackage{constants}

\makeatother

\usepackage{babel}
  \providecommand{\corollaryname}{Corollary}
  \providecommand{\lemmaname}{Lemma}
\providecommand{\theoremname}{Theorem}

\begin{document}

\title{Irreducible polynomials with several prescribed coefficients}

\author{Junsoo Ha}
\maketitle
\begin{abstract}
We study the number of irreducible polynomials over $\mathbf{F}_{q}$
with some coefficients prescribed. Using the technique developed by
Bourgain, we show that there is an irreducible polynomial of degree
$n$ with $r$ coefficients prescribed in any location when $r\leq\left[\left(1/4-\epsilon\right)n\right]$
for any $\epsilon>0$ and $q$ is large; and when $r\leq\delta n$
for some $\delta>0$ and for any $q$. The result improves earlier
work of Pollack stating that a similar result holds for $r\leq\left[(1-\epsilon)\sqrt{n}\right]$.
\end{abstract}

\section{Introduction and Statement of Result}

The problem of finding irreducible polynomials with certain properties
has been studied by numerous authors. One of the interesting problems
among them is the existence of an irreducible polynomial with certain
coefficients being prescribed. 

The early form of this problem is as follows. Let $\mathbf{F}_{q}$
be the finite field of $q$ elements and $n$ be a given integer,
and write a polynomial $P=\sum_{k\leq n}x_{k}T^{k}.$ Then we ask
if we can find an irreducible for any given pair of integers $j,$
$n$ and $a\in\mathbf{F}_{q}$ satisfying $x_{j}=a$, except when
$a=0$ and $j=0$. This problem, widely known as the Hansen-Mullen
conjecture, see \cite{HM}, has been settled by Wan \cite{wan} when
$n\geq36$ or $q>19$; the remaining cases were verified by Ham and
Mullen \cite{HamM}.

One may ask if we can find an irreducible with several preassigned
coefficients. In other words, we study the number of irreducible polynomials
of degree $n$ satisfying $x_{i}=a_{i}$ for all $i\in{\cal I}$,
when the index set ${\cal I}\subset\left\{ 0,1,\ldots,n-1\right\} $
and a finite sequence $a_{i}\in\mathbf{F}_{q}$ for $i\in{\cal I}$
are given. Unless we assume that the location of prescribed coefficients
has certain properties, the best known uniform bound is due to Pollack
\cite{pollack}, who proved that when $n$ is large, there is an irreducible
polynomial with $\left\lfloor (1-\epsilon)\sqrt{n}\right\rfloor $
prescribed coefficients. 

The analogue of the Hansen-Mullen conjecture in number theory is to
find rational primes with prescribed (binary) digits. Recently, Bourgain
\cite{bourgain} showed that for some $\delta>0$ and for large $n$,
there is a prime of $n$ digits with $\delta n$ digits prescribed
without any restriction on their position. Thus it is believed that
an analogous improvement holds for polynomials in finite fields. 

In this paper, we show that we can prescribe a positive proportion
of coefficients. The result presented here is the combination of several
known ideas. The underlying setup in this type of problems is the
circle method over $\mathbf{F}_{q}[T]$, which can be found in Hayes
\cite{hayes1966}. A recent application of this method, among others,
can be found in Liu and Wooley \cite{LW} on Waring's problem. The
work of Pollack \cite{pollack} is also implicitly based on this method.

The main element of this paper is the combination of Pollack's estimate
and the interpretation of the result of Bourgain \cite{bourgain}
in finite fields, though it is greatly simplified thanks to the Weil
bound, i.e., the analogue of the Riemann Hypothesis for irreducible
polynomials.

Our main theorem is as follows.
\begin{thm}
\label{thm:first}Let ${\cal I}$ be a nonempty subset of $\left\{ 0,\ldots,n-1\right\} $
and choose $a_{i}\in\mathbf{F}_{q}$ for each $i\in{\cal I}$. We
write as $\mathscr{S}$ the set of monic degree $n$ polynomials with
$T^{i}$ coefficient given by $a_{i}$ for each $i\in{\cal I}$. Then
if $\rho\coloneqq\left|{\cal I}\right|/n\leq1/4$,
\begin{equation}
\left(\sum_{\substack{P\in\mathscr{S}\\
P\text{ is irreducible}
}
}1\right)=\frac{\mathfrak{S}q^{n-\left|{\cal I}\right|}}{n}\left(1+O\left(\frac{\log_{q}\left(\frac{1}{\rho}\right)+1}{q^{1/\rho-4/(\rho+1)}}\right)\right)+O\left(q^{3n/4}\right),
\end{equation}
where the implied constants are absolute, and
\begin{equation}
\mathfrak{S}=\begin{cases}
1 & 0\notin{\cal I}\\
1+\frac{1}{q-1}\qquad & 0\in{\cal I}\text{ and }a_{0}\neq0\\
0 & 0\in{\cal I}\text{ and }a_{0}=0.
\end{cases}\label{eq:frakS.def}
\end{equation}
 
\end{thm}
Then the following corollary is a direct consequence of this theorem.
\begin{cor}
We have the following.
\begin{enumerate}
\item There is $\delta>0$ so that for any $q$, $n$, there is an irreducible
polynomial of degree $n$ with $[\delta n]$ prescribed coefficients,
unless the constant term is prescribed to 0.
\item For any $n\geq8$, $0<\epsilon<1/4$ and $q\geq q_{0}(\epsilon)$
for some large $q_{0}$, there is an irreducible polynomial of degree
$n$ with $[\left(1/4-\epsilon\right)n]$ prescribed coefficients,
unless the constant term is prescribed to 0.
\item When $n$ is large and $r=o(n)$, the number of irreducibles of degree
$n$ with $r$ prescribed coefficients is $\mathfrak{S}q^{n-r}(1+o(1))/n$
.
\end{enumerate}
\end{cor}
The implied constants can be explicitly computed and we conclude as
follows.
\begin{thm}
\label{thm:second}Suppose $n\geq8$, $q\geq16$, and $r\leq n/4-\log_{q}n-1$.
Then there exists monic irreducible polynomial of degree $n$ with
$r$ prescribed coefficients, except when 0 is assigned in the constant
term. We conclude the same when $q\geq5$, $n\geq97$, and $r\leq n/5$;
or when $n\geq52$, $r\leq n/10$ for arbitrary $q$.
\end{thm}

\subsection{Notation}

From now on, let $T$ be an indeterminate and we denote the ring of
polynomials over $\mathbf{F}_{q}$ by $\mathbf{F}_{q}[T]$. The polynomials
play a parallel role of integers in this paper, so we keep the polynomials
in $\mathbf{F}_{q}[T]$ as lowercase Latin letters whereas parameters
are usually written in capital letters. In particular, we substitute
the variable $n$ in \prettyref{thm:first} by $X$. Also, we use
$m$ for monic polynomials and $\varpi$ for monic irreducible polynomials.
The variable $g$ usually means the modulus, and is assumed to be
monic.

Following the setup of Hayes \cite{hayes1966}, we let $\mathbf{K}_{\infty}$
be the formal power series $\mathbf{F}_{q}((1/T))=\left\{ \sum_{i\ll\infty}a_{i}T^{i}\right\} $,
which is the completion of $\mathbf{F}_{q}[T]$ in the usual norm
\[
|m|=q^{\deg m}
\]
for polynomial $m$ (with convention $|0|=0$.) We extend this norm
to $\mathbf{K}_{\infty}$ by
\[
|x|=q^{L}
\]
where $L$ is the largest index so that $x_{L}\neq0$ (here and from
now on, whenever $x\in\mathbf{K}_{\infty}$, the subscripted $x_{k}$
denotes its $T^{k}$-coefficient.)

We define $\mathbf{T}$ by $\left\{ x\in\mathbf{K}_{\infty}\,:\,|x|<1\right\} $,
and fix an additive Haar measure, normalized so that $\int_{\mathbf{T}}dx=1$.
Finally we take a nontrivial additive character
\[
\mathbf{e}(x)=\exp\left(\frac{2\pi i}{p}\operatorname{tr}_{\mathbf{F}_{q}/\mathbf{F}_{p;}}(x_{-1})\right),
\]
where $p$ is the characteristic of $\mathbf{F}_{q}$. Then $\mathbf{e}(x)$
has similar property as $t\mapsto\exp(2\pi it)$ in number theory.
For one thing, we have for a polynomial $a\in\mathbf{F}_{q}[T]$,
\[
\int_{\mathbf{T}}\mathbf{e}(ax)dx=\begin{cases}
1\qquad & a=0\\
0 & a\neq0.
\end{cases}
\]
 We also adopt a convenient notation from Liu and Wooley \cite{LW}
that $\hat{X}=q^{X}$ and ${\cal L}(Z)=\max\left(\log_{q}Z,\,0\right)$.
For instance, $X={\cal L}(\hat{X})$ for $X\geq1$, and ${\cal L}(\left|m\right|)=\deg m$.
This is useful as we can write $|m|=\hat{X}$ in place of $\deg m=X$.
We also use $\pi_{q}(X)$ for the number of monic irreducible polynomials
with degree $X$. 

Now we define for $\alpha\in\mathbf{T}$,
\[
{\cal S}(\alpha)=\sum_{|\varpi|=\hat{X}}\mathbf{e}(\varpi\alpha)
\]
 and
\[
{\cal S}_{{\cal I}}(\alpha)=\sum_{m\in\mathscr{S}}\mathbf{e}(m\alpha)
\]
where $\mathscr{S}$ is as defined in \prettyref{thm:first}. Then
the number of irreducible polynomials in $\mathscr{S}$ is represented
by the integral
\begin{equation}
N=\int_{\mathbf{T}}{\cal S}(\alpha)\overline{{\cal S}_{{\cal I}}(\alpha)}d\alpha.\label{eq:N.def}
\end{equation}
We use $C_{i}$ to denote positive constants, which may depend on
many parameters but is always absolutely bounded. For instance, we
allow $C=q/(q-1)$, which is a constant depending on $q$ but is absolutely
bounded by 2; however, we do not allow $C=2^{q}$ as it is not absolutely
bounded. Due to their abundant appearance, we label
\begin{equation}
C_{(q^{B})}=\frac{q^{B}}{q^{B}-1}\label{eq:c.(q)}
\end{equation}
 for positive $B$ for the remainder of this paper, which is absolutely
bounded by $1+1/(2^{B}-1)$ if $B$ is bounded below by some positive
constant.

\section{Preliminaries}

The following lemmas are counterparts in $\mathbf{F}_{q}[T]$ for
well-known theorems in number theory.
\begin{lem}[Rational Approximation]
 \label{lem:rat.approx}For each $\alpha\in\mathbf{T}$, there exist
unique $a,$ $g\in\mathbf{F}_{q}[T]$ so that $g$ is monic, $|a|<|g|\leq\hat{X}^{1/2}$
and
\[
\left|\alpha-\frac{a}{g}\right|<\frac{1}{|g|\hat{X}^{1/2}}.
\]
\end{lem}
\begin{proof}
See Lemma 3 of \cite{pollack}
\end{proof}
We define a Farey arc
\begin{equation}
{\cal F}\left(\frac{a}{g},\hat{R}\right)=\left\{ \alpha\in\mathbf{T}\,:\,\left|\alpha-\frac{a}{g}\right|<\frac{1}{\hat{R}}\right\} .\label{eq:farey.arc}
\end{equation}
From \prettyref{lem:rat.approx}, we decompose $\mathbf{T}$ into
Farey arcs
\[
\mathbf{T}=\bigcup_{\left|a\right|<\left|g\right|\leq\hat{X}^{1/2}}{\cal F}\left(\frac{a}{g},\,|g|\hat{X}^{1/2}\right).
\]
The Farey arcs in the above decomposition are pairwise disjoint; to
prove, if $\alpha$ lies in two Farey arcs centered at distinct fractions
$a_{1}/g_{1}$ and $a_{2}/g_{2}$,
\begin{equation}
\frac{1}{\hat{X}^{1/2}\min(\left|g_{1}\right|,\left|g_{2}\right|)}>\max\left(\left|\alpha-\frac{a_{1}}{g_{1}}\right|,\left|\alpha-\frac{a_{2}}{g_{2}}\right|\right)\geq\left|\frac{a_{1}}{g_{1}}-\frac{a_{2}}{g_{2}}\right|\geq\frac{1}{\left|g_{1}g_{2}\right|}\label{eq:ultrametric}
\end{equation}
by the ultrametric inequality, which contradicts $\left|g_{i}\right|\leq\hat{X}^{1/2}$.

We define the set of major arcs by
\[
\mathfrak{M}\coloneqq\bigcup_{|a|<|g|\leq\hat{X}^{1/2}}{\cal F}\left(\frac{a}{g},\hat{X}\right)
\]
and the minor arcs by 
\[
\mathfrak{m}\coloneqq\mathbf{T}-\mathfrak{M}=\bigcup_{|a|<|g|\leq\hat{X}^{1/2}}{\cal F}\left(\frac{a}{g},\left|g\right|\hat{X}^{1/2}\right)-{\cal F}\left(\frac{a}{g},\hat{X}\right).
\]
 In Section \ref{sec:proof1}, we use the fact that ${\cal S}(\alpha)$
is small on minor arcs, and is well approximated on major arcs, and
that the most contribution of the integral \eqref{eq:N.def} came
from the major arcs.
\begin{lem}[Prime Number Theorem]
 \label{lem:pnt}We have
\[
\frac{\hat{X}}{{\cal L}(\hat{X})}-2\frac{\hat{X}^{1/2}}{{\cal L}(\hat{X})}\leq\pi_{q}(X)\leq\frac{\hat{X}}{{\cal L}(\hat{X})}.
\]
\end{lem}
\begin{proof}
See Lemma 4 of \cite{pollack}.

Let $\phi(m)$ be the number of reduced residue classes mod $m$,
i.e.,
\[
\phi(m)=\left|m\right|\prod_{\varpi|m}\left(1-\frac{1}{\left|\varpi\right|}\right).
\]
In number theory, Euler totient function $\varphi(n)$ has a lower
bound (see Theorem 2.9 of \cite{MV}) 
\[
\varphi(n)\geq e^{-\gamma}\frac{n}{\log\log n}\left(1+O\left(\frac{1}{\log\log n}\right)\right)
\]
for $n\geq3$. The similar estimate holds for $\phi(m)$ as well.\end{proof}
\begin{lem}
\label{lem:varphi.est}For $\deg m\leq q$ and $T\nmid m$, we have
\[
\frac{\left|m\right|}{\phi(m)}<e.
\]
If $\deg m>q$, and $T\nmid m$, we have
\[
\frac{\left|m\right|}{\phi(m)}\leq e^{\gamma}\left({\cal L}{\cal L}(\left|m\right|)+1\right).
\]
\end{lem}
\begin{proof}
The proof below is analogous to Theorem 2.9 of \cite{MV}. We write
$P_{A}=\prod_{\varpi\neq T,\left|\varpi\right|\leq\hat{A}}\varpi$
for a positive integer $A$, and we say $m$ is product of irreducibles
with smallest possible degrees if $P_{A}|m$ and $m|P_{A+1}$ for
some $A$, and let ${\cal R}$ be the set of polynomials $m$ satisfying
\begin{equation}
\frac{\left|m\right|}{\phi(m)}\geq\frac{\left|m_{1}\right|}{\phi(m_{1})}\label{eq:extreme.case.cond}
\end{equation}
 for any $m_{1}\in\mathbf{F}_{q}[T]$ such that $\left|m_{1}\right|<\left|m\right|$.

We claim that ${\cal R}$ contains only the polynomials that are products
of irreducibles with smallest possible degrees. If $m$ is a polynomial
with $k$ distinct prime factors, we take $m_{1}$ to be a polynomial
with $k$ prime factors that is the product of irreducibles of smallest
possible degrees. Then if $m$ is not the product of smallest possible
degrees, $\left|m\right|>\left|m_{1}\right|$ and 
\[
\frac{\left|m\right|}{\phi(m)}=\prod_{\varpi|m}\left(1-\frac{1}{\left|\varpi\right|}\right)^{-1}<\prod_{i\leq k}\left(1-\frac{1}{\left|\varpi_{i}\right|}\right)^{-1}=\frac{\left|m_{1}\right|}{\phi(m_{1})}
\]
where $\varpi_{i}$ is the choice of $k$ polynomials with smallest
possible degrees. Therefore $m\notin{\cal R}$ and an element of ${\cal R}$
is necessarily the product of smallest possible degrees.

We now show that it suffices to prove the lemma for $m\in{\cal R}$.
For simplicity let $f$ be the right hand side of the inequality,
i.e., $f(m)=e^{\gamma}\left({\cal L}{\cal L}(\left|m\right|)+1\right)$
if $\deg m>q$ and $e$ if $\deg m\leq q$. Note that $f$ is an increasing
function of $\left|m\right|$, since $2e^{\gamma}>e$. Suppose that
we proved the lemma for all polynomials in $\mathcal{R}$, and that
the lemma is false. Then there is a counterexample and let $m_{0}$
be the counterexample whose degree is smallest. From the assumption,
$m_{0}\notin{\cal R}$. Thus there is a polynomial $m_{1}$ such that
$|m_{1}|<|m_{0}|$ and $\left|m_{1}\right|/\phi(m_{1})>\left|m_{0}\right|/\phi(m_{0}).$
Then
\[
\frac{\left|m_{1}\right|}{\phi(m_{1})}>\frac{\left|m_{0}\right|}{\phi(m_{0})}>f(m_{0})\geq f(m_{1})
\]
and thus $m_{1}$ is also a counterexample for the lemma, which contradicts
the choice of $m_{0}$.

It remains to prove the lemma for $m\in\mathcal{R}$. Since each polynomial
in ${\cal R}$ is the product of irreducibles with smallest possible
degrees, we prove the lemma in this case.

When $m=P_{1}$, 
\[
\frac{m}{\phi(m)}=\prod_{\deg\varpi=1,\varpi\neq T}\left(1-\frac{1}{\left|\varpi\right|}\right)^{-1}=\left(1+\frac{1}{q-1}\right)^{q-1}<e,
\]
and thus all $\left|m_{1}\right|<\left|m\right|$ satisfies \eqref{eq:extreme.case.cond}.

When $m=P_{A}$ with $A>1$, 
\[
\frac{\left|m\right|}{\phi(m)}<e\prod_{1<r\leq A}\left(1+\frac{1}{q^{r}-1}\right)^{\pi_{q}(r)}.
\]
Since $r\pi_{q}(r)\leq q^{r}-1$, $\pi_{q}(r)\log\left(1+1/(q^{r}-1)\right)\leq1/r$
and thus
\[
\frac{\left|m\right|}{\phi(m)}<e^{\sum_{1\leq r\leq A}1/r}.
\]
From Euler-Maclaurin formula, we have
\[
\sum_{r\leq A}\frac{1}{r}\leq\log A+\gamma+\frac{1}{2A}
\]
and that for $0<x<3$, 
\[
e^{x}\leq1+x+\frac{x^{2}}{2(1-x/3)}.
\]
Combining these two estimate, we have
\[
\frac{\left|m\right|}{\phi(m)}<e^{\sum_{1\leq r\leq A}1/r}<A+\frac{1}{2}+\frac{1}{8(A-1/6)}.
\]
Since 
\[
\deg P_{A}=(-1)+\sum_{r\leq A}k\pi_{q}(k)\geq\sum_{k|A}k\pi_{q}(k)=q^{A}-1,
\]
we have 
\[
A\leq\log_{q}(\deg P_{A}+1)\leq{\cal L}{\cal L}\left(\left|P_{A}\right|\right)+\frac{1}{q^{A}-1}
\]
and thus $\left|m\right|/\phi(m)<e^{\gamma}\left(\mathcal{L}\mathcal{L}(\left|m\right|)+1\right)$.

Finally, for $m|P_{A}$ and $P_{A-1}|m$, we have
\[
\frac{\left|m\right|}{\phi(m)}=\frac{\left|P_{A-1}\right|}{\phi(P_{A-1})}\left(1+\frac{1}{q^{A}-1}\right)^{\deg m-\deg P_{A-1}}.
\]
Therefore we observe that $\log_{q}\left(\left|m\right|/\phi(m)\right)$
is linear in $\deg m$. To be precise, let $g(D)$ be the piecewise
linear continuous function defined on $D\geq\deg P_{1}$ whose breakpoints
are $D=\deg P_{A}$ for $A\geq1$ and satisfies $g(\deg P_{A})=\log\left(\left|P_{A}\right|/\phi(P_{A})\right)$.
From construction, we have $\log\left(\left|m\right|/\phi(m)\right)\leq g(\deg m)$.
Also, we have $g(D)\leq\log(\log_{q}D+1)$ on each breakpoint of $g$
and since $\log(\log_{q}D+1)$ is convex, we have $g(D)\leq\log(\log_{q}D+1)$
for any $D\geq P_{1}$. Therefore we conclude that
\[
\frac{\left|m\right|}{\phi(m)}\leq e^{g(\deg m)}\leq e^{\gamma}\left(\log_{q}\deg m+1\right)=e^{\gamma}\left({\cal L}{\cal L}(\left|m\right|)+1\right)
\]
as desired.
\end{proof}

\subsection{Analysis on ${\cal S}_{{\cal I}}$}

The norm of ${\cal S}_{{\cal I}}$ can be explicitly computed.
\begin{lem}
\label{lem:1.norm}We have
\[
\int_{\mathbf{T}}\left|{\cal S}_{{\cal I}}(\alpha)\right|d\alpha=1.
\]
\end{lem}
\begin{proof}
The set $\mathscr{S}$ can be rewritten as 
\[
\mathscr{S}=\left\{ m\,:\,|m|=\hat{X},\,m=T^{X}+\sum_{\substack{\substack{j\in{\cal I}}
}
}a_{j}T^{j}+\sum_{\substack{j\notin{\cal I}\\
j<X
}
}x_{j}T^{j}\text{ for some }x_{j}\in\mathbf{F}_{q}\right\} ,
\]
so we have
\begin{align}
{\cal S}_{{\cal I}}(\alpha) & =\mathbf{e}(\alpha T^{X})\prod_{j\in{\cal I}}\mathbf{e}(a_{j}T^{j}\alpha)\prod_{j\notin{\cal I}}\sum_{x_{j}\in\mathbf{F}_{q}}\mathbf{e}(x_{j}T^{j}\alpha)\nonumber \\
 & =\begin{cases}
q^{X-|{\cal I}|}\mathbf{e}(\alpha T^{X})\prod_{j\in{\cal I}}\mathbf{e}(a_{j}T^{j}\alpha)\qquad & \alpha_{-j-1}=0\text{ for all }j\notin{\cal I}\\
0 & \text{otherwise.}
\end{cases}\label{eq:S_I.simplified}
\end{align}
From the definition, $\left|{\cal S}_{{\cal I}}(\alpha)\right|$ depends
only on the first $X$ coefficients of Laurent series expansion, and
thus it is constant on the range $\left|\alpha-a/T^{X}\right|<1/\hat{X}$
for each polynomial $a\in\mathbf{F}_{q}[T]$. Therefore
\[
\int_{\mathbf{T}}\left|{\cal S}_{{\cal I}}(\alpha)\right|d\alpha=\frac{1}{\hat{X}}\sum_{|a|<\hat{X}}\left|{\cal S}_{{\cal I}}\left(\frac{a}{T^{X}}\right)\right|=q^{-\left|{\cal I}\right|}\sum_{\substack{|a|<\hat{X}\\
a_{X-j}=0\,\forall j\notin{\cal I}
}
}1=1
\]
which proves the lemma.
\end{proof}
We need the following covering lemma to apply Bourgain's technique
where he simply used $\kappa=2$ in the theorem below. We slightly
improve the constant so that we may apply when the density $\left|{\cal I}\right|/X$
is close to $1/4$. 
\begin{lem}[Covering Lemma]
\label{lem:covering} Let $x$, $y$ be integers with $1\leq y\leq x$
and $I\subseteq[1,x]$ be a given set of integers. Then there exists
a set of consecutive integers $J$ of length $y$ so that 
\[
\frac{\left|I\cap J\right|}{\left|J\right|}\leq\kappa\rho
\]
where $\rho=\left|I\right|/x$ and $\kappa\leq2$ is given by
\[
\kappa(x,\,y,\,\rho)=\begin{cases}
\frac{2}{\rho+1}\qquad & 1<x/y<2\\
\frac{2u}{(u+1)\rho+(u-1)}\qquad & u-1<x/y<u\\
1 & y|x.
\end{cases}
\]
\end{lem}
\begin{proof}
If $x$ is multiple of $y$, the result is trivial because $[1,x]$
can be covered by nonoverlapping intervals of length $y$, and by
the box principle, at least one subinterval, say $J$ satisfies the
density $\left|I\cap J\right|/\left|J\right|\leq\rho$; so we assume
otherwise. We write $\left|I\right|=z$ and $u=\left\lceil x/y\right\rceil \geq2$.
Now, we cover $[1,x]$ into $u$ intervals of length $y$, say $J_{1}$,
$\cdots$, $J_{u}$ where the smallest element of each $J_{i}$ is
$[(i-1)x/u]$. Then we set $\kappa_{0}$ to satisfy 
\[
\kappa_{0}\rho=\frac{1}{y}\max_{I}\min_{1\leq i\leq u}\left\{ \left|I\cap J_{i}\right|\right\} .
\]
The exact formula for $\kappa_{0}$ is a bit complicated as $y$,
$z$ vary, but we can find some upper bound, and any upper bound for
$\kappa_{0}$ would work for $\kappa$ in our lemma.

If $u=2$, as $I$ varies, the minimum density $\min_{i}\left|I\cap J_{i}\right|/y$
gets largest when the intersection of $I$ and $J_{1}\cap J_{2}$
is as large as possible. Let $\rho_{y}=(2y-x)/x$, which is the density
of the overlapping interval out of the total length. If $\rho\leq\rho_{y}$,
we get the trivial bound
\[
\kappa_{0}\rho=\frac{\left|I\right|}{y}=\frac{2\rho}{\rho_{y}+1}.
\]
If $\rho>\rho_{y}$, the minimum density gets largest when we take
$I$ to fill the overlap and equally split the remaining to $J_{1}-J_{2}$
and $J_{2}-J_{1}$; then
\[
\kappa_{0}\rho=\left[\frac{\left|I\right|+\left|J_{1}\cap J_{2}\right|}{2}\right]\leq\frac{\rho+\rho_{y}}{\rho_{y}+1}.
\]
Thus in either case, $\kappa=2/\left(\rho+1\right)$ works.

If $u\geq3$, the minimum density gets largest when the intersection
of $I$ and $\bigcup\left(J_{i}\cap J_{i+1}\right)$ is as large as
possible; that is when $\left|I\right|$ is small, $I$ intersects
each $J_{i}\cap J_{i+1}$ and the two tails $J_{1}-J_{2},$ $J_{u}-J_{u-1}$
equally by $\left|I\right|/(u+1)$, and when $\left|I\right|$ is
large, $I$ covers all overlapping intervals and distribute remaining
so that $I$ intersects each $J_{i}$ by almost equal length. To compute,
let $\rho_{y}=\left(uy-x\right)/x$. If $\rho\leq(u+1)\rho_{y}/(u-1)$,
\[
\kappa_{0}\rho=\frac{2\left|I\right|/(u+1)}{y}=\frac{2\rho u}{(u+1)(\rho_{y}+1)}\leq\frac{2u}{(u-1)\rho+(u+1)}\rho
\]
If $\rho>(u+1)\rho_{y}/(u-1)$,
\begin{align*}
\kappa_{0}\rho & \leq\frac{\frac{1}{u}\left(\left|I\right|-\frac{u+1}{u-1}\left|\bigcup\left(J_{i}\cap J_{i+1}\right)\right|\right)+\frac{2}{u-1}\left|\bigcup\left(J_{i}\cap J_{i+1}\right)\right|}{y}\\
 & =\frac{\rho+\rho_{y}}{\rho_{y}+1}\leq\frac{2u}{(u-1)\rho+(u+1)}\rho
\end{align*}
which proves the lemma.
\end{proof}
The following lemma, whose analogue in number theory is found in Lemma
3 of Bourgain \cite{bourgain}, is the key advantage of this paper. 
\begin{lem}
\label{lem:bourgain.est}Let $Q$ be such that $\hat{Q}^{2}\leq\hat{X}$.
We have
\[
\sum_{\substack{|g|<\hat{Q}\\
(aT,g)=1
}
}\left|{\cal S}_{{\cal I}}\left(\frac{a}{g}\right)\right|\leq\hat{X}q^{-\left|{\cal I}\right|}\hat{Q}^{2\Cl{cover}\left|{\cal I}\right|/X}
\]
where $\Cr{cover}=\kappa(X,2Q,\left|\mathcal{I}\right|/X)$ and $\kappa$
is as defined in \prettyref{lem:covering}. \end{lem}
\begin{proof}
Any two fractions are pairwise separated by the norm of size $\hat{Q}^{-2}$
by ultrametric inequality (see \eqref{eq:ultrametric}), and thus
the arcs ${\cal F}(a/g,\hat{X})$ are pairwise disjoint. On the other
hand, each term of ${\cal S}_{{\cal I}}$ is $\mathbf{e}(m\alpha)$
with monic polynomial $m$, $\mathbf{e}(m\alpha)\mathbf{e}(-T^{X}\alpha)$
remains constant when $\alpha$ varies in norm of size $<1/\hat{X}$
and thus $\left|{\cal S_{{\cal I}}}\right|$ is constant on ${\cal F}(\alpha,\hat{X})$
for each $\alpha\in\mathbf{T}$. Therefore
\begin{equation}
\hat{X}^{-1}\sum_{\substack{|a|<|g|<Q\\
(aT,g)=1
}
}\left|{\cal S}_{{\cal I}}\left(\frac{a}{g}\right)\right|=\sum_{a,g}\int_{{\cal F}(a/g,\hat{X})}\left|{\cal S}_{{\cal I}}(\alpha)\right|d\alpha\leq\int_{\mathbf{T}}\left|{\cal S}_{{\cal I}}(\alpha)\right|d\alpha=1.\label{eq:sum.fractions}
\end{equation}

We write for any integer $X_{1}\leq X$, an index set ${\cal I}_{1}\subseteq\left\{ 0,\ldots,X_{1}-1\right\} $
and a finite sequence $a_{j}\in\mathbf{F}_{q}$ for $j\in{\cal I}_{1}$,
\[
{\cal S}_{{\cal I}_{1}}^{(X_{1})}(\alpha)=\sum_{m}\mathbf{e}(m\alpha)
\]
where $m$ runs over all monic polynomials of degree $X_{1}$ whose
$T^{j}$-coefficient is $a_{j}$ for any $j\in{\cal I}_{1}$, to emphasize
the dependency on $X_{1}$. As we can see in \eqref{eq:S_I.simplified},
$\left|{\cal S}_{{\cal I}_{1}}^{(X_{1})}(\alpha)\right|$ does not
depend on the choice of $a_{j}\in\mathbf{F}_{q}$ and since we only
use it with the absolute value, we do not write $a_{j}$ for simplicity.
Following \eqref{eq:sum.fractions}, we have for any integer $Q$
and any index set ${\cal I}_{1}\subseteq\left\{ 0,\ldots,2Q-1\right\} $,
\[
\frac{1}{\hat{Q}^{2}}\sum_{\substack{|a|<|g|<Q\\
(aT,g)=1
}
}\left|{\cal S}_{{\cal I}}^{(2Q)}\left(\frac{a}{g}\right)\right|\leq1.
\]

Now, we take a subset ${\cal J}$ of $\left\{ 0,\cdots,X-1\right\} $
consisting of consecutive numbers so that the length of the interval
is $2Q$ and its intersection with ${\cal I}$ is of size 
\[
\left|{\cal I}\cap{\cal J}\right|\leq\Cr{cover}\frac{\left|{\cal I}\right|}{X}(2Q).
\]
for $\Cr{cover}=\kappa(X,2Q,\left|I\right|/X)$ by \prettyref{lem:covering}.
We write ${\cal J}=\left\{ j_{\ast},\ldots,j_{\ast}+2Q-1\right\} $
for some $j_{\ast}$.

We put ${\cal I}^{\prime}=(-j_{\ast})+{\cal I}\cap{\cal J}$. Then
we relate ${\cal S}_{{\cal I}}^{(X)}$ with ${\cal S}_{{\cal I}^{\prime}}^{(2Q)}$
by 
\begin{align*}
\left|{\cal S}_{{\cal I}}^{(X)}(\alpha)\right| & =q^{X-\left|{\cal I}\right|}\cdot\mathbf{1}\left\{ \alpha_{-j-1}=0\text{ for all }0\leq j\leq X\text{ and }j\notin{\cal I}\right\} \\
 & \leq q^{X-\left|{\cal I}\right|}\cdot\mathbf{1}\left\{ \alpha_{-j_{\ast}-j-1}=0\text{ for all }0\leq j\leq2Q\text{ and }j\notin{\cal I}^{\prime}\right\} \\
 & =q^{X-2Q}q^{\left|{\cal I}^{\prime}\right|-\left|{\cal I}\right|}\left|{\cal S}_{{\cal I}^{\prime}}^{(2Q)}\left(T^{j_{\ast}}\alpha\right)\right|.
\end{align*}
Therefore we apply \eqref{eq:sum.fractions} on ${\cal S}_{{\cal I}^{\prime}}$
to conclude
\begin{align*}
\sum_{\substack{\left|g\right|<\hat{Q}\\
(aT,g)=1
}
}\left|{\cal S}_{{\cal I}}^{(X)}\left(\frac{a}{g}\right)\right| & \leq q^{X-2Q+\left|{\cal I}^{\prime}\right|-\left|{\cal I}\right|}\sum_{\substack{\left|g\right|<\hat{Q}\\
(aT,g)=1
}
}\left|{\cal S}_{{\cal I}^{\prime}}^{(2Q)}\left(\frac{T^{j_{\ast}}a}{g}\right)\right|\\
 & =q^{X-2Q+\left|{\cal I}^{\prime}\right|-\left|{\cal I}\right|}\sum_{a,g}\left|{\cal S}_{{\cal I}^{\prime}}^{(2Q)}\left(\frac{a}{g}\right)\right|\\
 & \leq q^{-\left|{\cal I}\right|}\hat{X}\hat{Q}^{2\Cr{cover}\left|{\cal I}\right|/X}.
\end{align*}

\end{proof}
The next lemmas are similar to Lemma 6 and 7 of \cite{pollack}.
\begin{lem}
\label{lem:pollack.est}Let $a$, $g\in\mathbf{F}_{q}[T]$ be two
given polynomials with $(a,g)=1$, and $g_{0}$ be such that $g=g_{0}T^{k}$
with $(g_{0},T)=1$. If $1<\left|g_{0}\right|\leq q^{\lceil X/(\left|{\cal I}\right|+1)\rceil-1}$
\[
{\cal S}_{{\cal I}}\left(\frac{a}{g}\right)=0
\]
\end{lem}
\begin{proof}
Suppose ${\cal S}_{{\cal I}}(a/g)\neq0$. Then in the Laurent series
expansion of $a/g$, the $T^{-j-1}$ coefficient vanishes for any
$j\notin{\cal I}$ and $0\leq j<X$ from \eqref{eq:S_I.simplified}.
We write $J=\lceil X/(\left|{\cal I}\right|+1)\rceil$. Then by the
box principle, there is at least $\lceil(X-\left|{\cal I}\right|)/(\left|{\cal I}\right|+1)\rceil\geq J-1$
consecutive indices where the Laurent series of $a/g$ vanishes.

We now show that if the Laurent series of $a/g$ has $J-1$ consecutive
zeros and write $g=g_{0}T^{k}$ with $(g_{0},T)=1$, then $\left|g_{0}\right|\geq\hat{J}$,
which will prove the lemma. If the Laurent series has $J$ consecutive
zeros, we shift the series by multiplying some power of $T$ to have
\[
\left|\left\{ \frac{T^{r}a}{g}\right\} \right|\leq\left|\frac{1}{T^{J}}\right|\leq\frac{1}{\hat{J}}.
\]
 for some integer $r$, where $\left\{ x\right\} $ denotes the fractional
part of $x$. However unless $g_{0}=1$, the left hand side is at
least $\left|g_{0}\right|^{-1}$. Therefore $\left|g_{0}\right|\geq\hat{J}$
as desired.
\end{proof}

\subsection{Analysis on ${\cal S}(\alpha)$}

The analysis on ${\cal S}(\alpha)$ is fairly standard. We cite the
following estimate as in Lemma 5 of \cite{pollack}. We cite the original
result due to \cite{hayes1966}, which is slightly more precise.
\begin{lem}
\label{lem:exp.sum}Let $a$, $g\in\mathbf{F}_{q}[T]$ be two polynomials
with $(a,g)=1$ and $\gamma\in\mathbf{T}$, satisfying $\left|a\right|<\left|g\right|<\hat{X}^{1/2}$
and $\left|\gamma\right|<1/\left|g\right|\hat{X}^{1/2}$. We have
\[
{\cal S}\left(\frac{a}{g}+\gamma\right)=\frac{\mu(g)}{\phi(g)}\pi_{q}(X)\mathbf{e}(\gamma T^{X})\mathbf{1}_{|\gamma|<1/\hat{X}}+R
\]
with 
\[
\left|R\right|\leq\sqrt{\phi(g)\max(1,\left|\gamma T^{X}\right|)}\hat{X}^{1/2}\leq\hat{X}^{3/4}.
\]
\end{lem}
\begin{proof}
See Lemma 5 of \cite{pollack} and (5.14) in Theorem 5.3 of \cite{hayes1966}.
\end{proof}

\section{\label{sec:proof1}Proof of Theorem \ref{thm:first}}

Let $\mathfrak{M}=\bigcup\mathcal{F}(a/g,\hat{X})$ where the union
is taken over fractions $a/g$ with $\left|g\right|\leq\hat{X}^{1/2}$,
and $\mathfrak{m}=\mathbf{T}-\mathfrak{M}$. Then from \prettyref{lem:exp.sum},
$\max_{\alpha\in\mathfrak{m}}\left|{\cal S}(\alpha)\right|\leq\hat{X}^{3/4}$.

Recall that the number of irreducible polynomials with prescribed
coefficients is given by the integral
\[
N=\int_{\mathbf{T}}{\cal S}(\alpha)\overline{{\cal S}_{{\cal I}}(\alpha)}d\alpha.
\]
Then we have
\begin{equation}
\left|N-\int_{\mathfrak{M}}{\cal S}(\alpha)\overline{{\cal S}_{{\cal I}}(\alpha)}d\alpha\right|\leq\max_{\alpha\in\mathfrak{m}}\left|{\cal S}(\alpha)\right|\int_{\mathbf{T}}\left|{\cal S}_{{\cal I}}(\alpha)\right|d\alpha\label{eq:int.major.minor}
\end{equation}
and the right hand side is bounded by $\hat{X}^{3/4}$ using \prettyref{lem:exp.sum}
and \prettyref{lem:1.norm}. We recall that all $\varpi$ and $m$
appearing in the sums ${\cal S}_{{\cal I}}(\alpha)$ and ${\cal S}(\alpha)$
are monic, and thus $\mathbf{e}(m\gamma)=\mathbf{e}(\gamma T^{X})$
for $|\gamma|<1/\hat{X}$. Therefore we have for $\left|\gamma\right|<1/\hat{X}$
\[
{\cal S}\left(\frac{a}{g}+\gamma\right)={\cal S}\left(\frac{a}{g}\right)\mathbf{e}(\gamma T^{X})
\]
and similarly for ${\cal S}_{{\cal I}}(a/g+\gamma)$. Then the main
term can be written as
\begin{align}
\int_{\mathfrak{M}}{\cal S}(\alpha)\overline{{\cal S}_{{\cal I}}(\alpha)}d\alpha & =\sum_{a,g}\int_{\left|\gamma\right|<1/\hat{X}}{\cal S}\left(\frac{a}{g}+\gamma\right)\overline{S_{{\cal I}}\left(\frac{a}{g}+\gamma\right)}d\gamma\nonumber \\
 & =\frac{1}{\hat{X}}\sum_{a,g}{\cal S}\left(\frac{a}{g}\right)\overline{{\cal S}_{{\cal I}}\left(\frac{a}{g}\right)}
\end{align}
where the sum is taken over distinct fractions with $|g|\leq\hat{X}^{1/2}$.

We expect the main term of the integral to be
\begin{align}
M & \coloneqq\frac{1}{\hat{X}}\left({\cal S}(0)\overline{{\cal S}_{{\cal I}}(0)}+\sum_{b\in\mathbf{F}_{q}^{\ast}}{\cal S}\left(\frac{b}{T}\right)\overline{{\cal S}_{{\cal I}}\left(\frac{b}{T}\right)}\right).\label{eq:M.def}
\end{align}
If $0\notin{\cal I}$, ${\cal S}_{{\cal I}}(b/T)=0$ and thus $M=q^{-\left|{\cal I}\right|}\pi_{q}(X)$.
If $0\in{\cal I}$, we have by \prettyref{lem:exp.sum}, 
\[
{\cal S}\left(\frac{b}{T}\right)=\frac{\pi_{q}(X)}{\phi(T)}+O\left(\sqrt{\phi(T)}\hat{X}^{1/2}\right)
\]
where the implied constant is bounded by 1, and ${\cal S}_{{\cal I}}(b/T)=\mathbf{e}(a_{0}b/T)\pi_{q}(X)q^{-\left|{\cal I}\right|}$.
Then

\begin{align*}
M & =\frac{\pi_{q}(X)}{q^{\left|{\cal I}\right|}}\left(1+\frac{\mu(T)}{\phi(T)}\sum_{a\in\mathbf{F}_{q}^{\ast}}\overline{e\left(\frac{a_{0}b}{T}\right)}\right)+O\left(\frac{\sqrt{q}\pi_{q}(X)}{q^{\left|{\cal I}\right|}\hat{X}^{1/2}}\right)\\
 & =\begin{cases}
O\left(q^{1/2-\left|{\cal I}\right|}\hat{X}^{1/2}\right) & a_{0}=0\\
\left(1+\frac{1}{q-1}\right)\frac{\pi_{q}(X)}{q^{\left|{\cal I}\right|}}+O\left(q^{1/2-\left|{\cal I}\right|}\hat{X}^{1/2}\right)\qquad & a_{0}\neq0,
\end{cases}
\end{align*}
where the implied constants are bounded by 1. Thus we have
\[
M=\mathfrak{S}\frac{\pi_{q}(X)}{q^{\left|{\cal I}\right|}}+O\left(q^{1/2-\left|{\cal I}\right|}\hat{X}^{1/2}\right)
\]
where $\mathfrak{S}$ is defined in \eqref{eq:frakS.def}. It is not
hard to replace $\pi_{q}(X)$ by $\hat{X}/{\cal L}(\hat{X})$ with
a small error by \prettyref{lem:pnt}.

Now we consider the remaining terms. Let $J=\left\lceil \frac{X}{|{\cal I}|+1}\right\rceil $.
The remaining terms are
\[
\sum_{\left|g\right|>1}\frac{\mu(g)^{2}}{\phi(g)}\pi_{q}(x)\sum_{(a,g)=1}\left|{\cal S}_{{\cal I}}(a/g)\right|.
\]
The terms with $|g|>1$ and $(g,T)=1$ contribute
\[
\sum_{\left|g\right|>1}\frac{\mu(g)^{2}}{\phi(g)}\sum_{(a,g)=1}\left|{\cal S}_{{\cal I}}\left(\frac{a}{g}\right)\right|=\sum_{\left|g\right|\geq\hat{J}}\frac{\mu(g)^{2}}{\phi(g)}\sum_{(a,g)=1}\left|{\cal S}_{{\cal I}}\left(\frac{a}{g}\right)\right|
\]
because ${\cal S}_{{\cal I}}(a/g)=0$ for $\left|g\right|<\hat{J}$
from \prettyref{lem:pollack.est}. Now we assume $\left|{\cal I}\right|\leq X/4$,
and use $\Cl{exp}=1-2\Cr{cover}\left|{\cal I}\right|/X$; when $\left|{\cal I}\right|\leq X/4$,
$\Cr{exp}\leq1/5$.

We apply the estimate in \prettyref{lem:varphi.est}, and group the
fractions $a/g$ according to the degree of $g$. Combined with \prettyref{lem:bourgain.est},
when $J<q$,
\[
\frac{1}{q^{X-\left|{\cal I}\right|}}\sum_{|g|\geq\hat{J}}\frac{\mu^{2}(g)}{\phi(g)}\sum_{a}\left|{\cal S}_{{\cal I}}\left(\frac{a}{g}\right)\right|\leq e\sum_{n\geq J}\frac{1}{q^{\Cr{exp}n}}+e^{\gamma}\sum_{n\geq q}\frac{\log_{q}n}{q^{\Cr{exp}n}}.
\]
We have, for any integer $A\geq1$, 
\begin{align}
\sum_{n\geq A}\frac{\log_{q}n}{q^{\Cr{exp}n}} & \leq\left(\log_{q}A+\frac{1}{A\log q}\right)\sum_{n\geq A}q^{-\Cr{exp}n}\nonumber \\
 & \leq\left(\log_{q}A+\frac{1}{A}\right)C_{(q^{\Cr{exp}})}q^{-A}.\label{eq:sumprod.est}
\end{align}
Using this formula, we have
\[
\sum_{n\geq J}\frac{\log_{q}n}{q^{\Cr{exp}n}}\leq\frac{\Cl{J.c1}}{q^{\Cr{exp}J}}
\]
where $\Cr{J.c1}=C_{\left(q^{\Cr{exp}}\right)}\left(e+e^{\gamma}+e^{\gamma}/q\log q\right)$. 

If $J\geq q$, we use \eqref{eq:sumprod.est} to obtain
\begin{align*}
\frac{1}{q^{X-\left|{\cal I}\right|}}\sum_{\left|g\right|\geq\hat{J}}\frac{\mu^{2}(g)}{\phi(g)}\left|{\cal S}_{{\cal I}}\left(\frac{a}{g}\right)\right| & \leq e^{\gamma}C_{\left(q^{\Cr{exp}}\right)}\frac{\log_{q}J+1+1/J\log q}{q^{\Cr{exp}J}}\\
 & =\frac{\Cl{J.c2}\log_{q}J+\Cl{J.c3}}{q^{\Cr{exp}J}}.
\end{align*}
 Thus we have
\begin{align}
\frac{1}{q^{X-\left|{\cal I}\right|}} & \sum_{\left|g_{0}\right|\geq\hat{J}}\left(\frac{\mu^{2}(g_{0})}{\phi(g_{0})}+\frac{\mu^{2}(Tg_{0})}{\phi(Tg_{0})}\right)\sum_{(a,g)=1}\left|{\cal S}_{{\cal I}}(a/g)\right|\nonumber \\
 & \leq\frac{\Cl{J.c4}{\cal L}{\cal L}(\hat{J})+\Cl{J.c5}}{\hat{J}^{\Cr{exp}}}\label{eq:sums.error.final}
\end{align}
for some $\Cr{J.c4}$ and $\Cr{J.c5}$.

Therefore, the integral is
\[
N=M+R_{1}=\mathfrak{S}\frac{\pi_{q}(X)}{q^{\left|{\cal I}\right|}}+R_{1}+R_{2}
\]
with 
\[
\left|R_{1}\right|\leq\frac{\pi_{q}(X)}{q^{\left|{\cal I}\right|}}\cdot\frac{\Cr{J.c4}{\cal L}(J)+\Cr{J.c5}}{q^{\Cr{exp}J}}
\]
 and $\left|R_{2}\right|\leq\hat{X}^{3/4}+q^{1/2-\left|{\cal I}\right|}\hat{X}^{1/2}$.
These errors and the replacement of $\pi_{q}(X)$ by $\hat{X}/{\cal L}(\hat{X})$
are absorbed in $O(\hat{X}^{3/4})$, which proves \prettyref{thm:first}.

\section{Evaluation of Constants and Proof of Theorem \ref{thm:second}}

We continue from the previous section. Let $B$ be a constant to be
specified later. We need to show the integral \eqref{eq:N.def} is
positive when $q^{-\left|{\cal I}\right|}\pi_{q}(X)\geq\left|R_{1}\right|+\left|R_{2}\right|$.
We assume that $\left|{\cal I}\right|\leq X/4-\log_{q}X-B$. Then
we have $\hat{X}^{3/4}\leq\pi_{q}(X)X^{-\left|{\cal I}\right|}q^{-B}$,
so the sufficient condition is 
\[
1>\frac{\Cl{q.1}}{q^{\Cr{exp}/\rho}}+\frac{1}{q^{B}}+O\left(\frac{\sqrt{q}\hat{X}^{1/2}}{\pi_{q}(X)}\right).
\]
where $\Cr{q.1}=C_{(q)}\max\left\{ \Cr{J.c1},\Cr{J.c2}\log_{q}(1/\rho)+\Cr{J.c3}\right\} $.
The big-O term is numerically tiny and we exclude this from computation.
Computing other constants, one can show for $q\geq16$ and $B=1$,
$\Cr{q.1}\leq8.552$ and the right hand size is $\leq0.994$. when
$q\geq5$, $\Cr{q.1}\leq11.684$ and for $\left|{\cal I}\right|\leq X/5$,
and $B=2$, $\Cr{q.1}\leq12.335$ and the right hand side is $\leq0.884$.
If $\left|{\cal I}\right|\leq X/10$ and $B=2$, $\Cr{q.1}\leq42.342$
and the right hand side is $\leq0.764$. The second case is when $X\geq97$
and the last case is when $X\geq52$, which proves \prettyref{thm:second}.

\section*{Acknowledgement}

This article was revised from earlier draft of the author. The author
like to thank the anonymous reviewers for carefully reading the article
and suggesting helpful advice.

\section*{}

\end{document}